\documentclass[preprint,12pt]{elsarticle}
\usepackage{amsmath,amssymb,amsthm}

\newtheorem{thr}{Theorem}[section]
\newtheorem{lem}[thr]{Lemma}
\newtheorem{cor}[thr]{Corollary}
\newtheorem{obs}[thr]{Observation}

\theoremstyle{definition}

\numberwithin{equation}{section}

\def\lk{\operatorname{left\,ker}}

\def\li{\operatorname{left\,im}}
\def\ri{\operatorname{right\,im}}

\journal{}

\begin{document}

\begin{frontmatter}

\title{Which semifields are exact?}

\author{Yaroslav Shitov}

\ead{yaroslav-shitov@yandex.ru}

\address{National Research University Higher School of Economics, 20 Myasnitskaya Ulitsa, Moscow 101000, Russia}

\begin{abstract}
Every (left) linear function on a subspace of a finite-dimensional vector space over a (skew) field can be extended to a (left) linear function on the whole space. This paper explores the extent to what this basic fact of linear algebra is applicable to more general structures. Semifields with a similar property imposed on linear functions are called (left) exact, and we present a complete description of such semifields. Namely, we show that a semifield $S$ is left exact if and only if $S$ is either a skew field or an idempotent semiring. 
\end{abstract}

\begin{keyword}
exact semiring \sep idempotent semiring \sep semifield

\MSC[2010] 16Y60 \sep 15A80

\end{keyword}

\end{frontmatter}

\section{Introduction}

A set $S$ equipped with two binary operations $+$ and $\cdot$ is called a \textit{semiring} if the following conditions are satisfied: (i) $(R,+)$ is a commutative monoid, (ii) $(R,\cdot)$ is a monoid, (iii) multiplication distributes over addition from both sides, and (iv) the additive identity $0$ satisfies $0x=x0=0$, for any $x\in R$. In other words, semirings differ from rings by the fact that their elements are not required to have additive inverses. We denote the multiplicative identity by $1$, and we assume that $0\neq 1$. The set $S^n$ becomes a \textit{free left semimodule} if we define the operations $(s_1,\ldots,s_n)\to (\lambda s_1,\ldots,\lambda s_n)$ for all $\lambda\in S$.

A considerable amount of recent work~\cite{Shi,Wil,WJK} is devoted to the concept of so-called \textit{exactness}, which gives a characterization of semirings that behave nicely with respect to basic linear algebraic properties. Namely, a semiring $S$ is called \textit{left exact} if, for every finitely generated left semimodule $L\subseteq S^n$ and every left $S$-linear function $\varphi: L\to S$, there is a left $S$-linear function $\varphi_0: S^n\to S$ that coincides with $\varphi$ on $L$. This property becomes a standard result of linear algebra if $S$ is a division ring, so we can conclude that division rings are left exact. The concept of right exactness can be defined dually, and the semirings that are both left and right exact are called simply \textit{exact}. Therefore, the division rings are the first examples of exact semirings. Let us also point out that, in the case of rings, the exactness is equivalent to the property known as \textit{FP-injectivity}, see~\cite{Gar, JN, Wil}.

In this paper, we continue studying the semirings in which all the non-zero elements have multiplicative inverses. Such objects form an important class of semirings and are known as \textit{semifields}. Various examples of semifields arise in different applications, and they include the division rings, the semiring of nonnegative reals~\cite{Yan}, the tropical semiring~\cite{SpSt}, the binary Boolean algebra~\cite{RS}, and many others. The aim of our paper is to give a complete characterization of those semifields that are exact.

\begin{thr}\label{thrthr}
Let $S$ be a semifield. Then $S$ is left exact if and only if

\noindent (1) $S$ is a division ring, or

\noindent (2) we have $1+1=1$ in $S$.
\end{thr}

By symmetry, the conclusion of the theorem holds for right exactness as well. In particular, we get that a semifield is left exact if and only if it is right exact. As a corollary of Theorem~\ref{thrthr}, we get the previously known fact that the tropical semiring $\mathbb{T}=(\mathbb{R}\cup\{\infty\},\min,+)$ is exact. As far as I can see, the exactness of $\mathbb{T}$ follows from Theorem~5.3 of~\cite{LMS}, and I would like to thank the reviewer for pointing my attention to that paper. Corollary~40 in~\cite{CSQ} contains a generalization of this result to the class of complete idempotent reflexive semirings. The subsequent paper~\cite{WJK} contains the exactness proofs for other related semirings, including $\overline{\mathbb{T}}=(\mathbb{R}\cup\{+\infty,-\infty\},\min,+)$. Another proof that $\mathbb{T}$ is exact is contained in the paper~\cite{JN}, where the authors do also prove that an additively cancellative semifield is exact if and only if it is a field.

Our paper is structured as follows. In Section~2, we obtain a useful characterization of exactness resembling some of the results in~\cite{WJK}. We use this characterization (Theorem~\ref{lemlem1}) to prove the 'only if' part of Theorem~\ref{thrthr}. In Section 3, we get an improved version of Theorem~\ref{lemlem1} which is valid for semifields. In Section~4, we employ the developed technique and complete the proof of Theorem~\ref{thrthr}. In Section~5, we discuss the perspectives of further work and point out several intriguing open questions.

\section{Another characterization of exactness}

Let us begin with some notational conventions. We will denote matrices and vectors over a semiring $S$ by bold letters. We denote by $\mathbf{A}_i$ and $\mathbf{A}^j$ the $i$th row and $j$th column of a matrix $\mathbf{A}$, and by $\mathbf{A}_i^j$ the entry at the intersection of the $i$th row and $j$th column. By $\mathbf{E}$ we denote the unit matrices, that is, square matrices with ones on the diagonal and zeros everywhere else. In particular, $\mathbf{E}_i$, $\mathbf{E}^i$ stand for the $i$th unit row and column vectors, respectively. Every $d\times n$ matrix induces the function $S^{1\times d}\to S^{1\times n}$ defined as $\mathbf{u}\to \mathbf{u}\mathbf{A}$. We denote the image of this operator by $\operatorname{left\,im} \mathbf{A}$, and the kernel of this operator as $\operatorname{left\,ker} \mathbf{A}$. In other words, $\operatorname{left\,ker} \mathbf{A}$ is the set of all pairs $(\mathbf{u},\mathbf{v})\in S^{1\times d}\times S^{1\times d}$ such that $\mathbf{u}\mathbf{A}=\mathbf{v}\mathbf{A}$. The right image and right kernel of $\mathbf{A}$ are defined dually in a natural way. In particular, we define $\ri \mathbf{A}$ as the set of all vectors in $S^{d\times 1}$ that can be written as $\mathbf{A}\mathbf{w}$ with some $\mathbf{w}\in S^{n\times1}$. We proceed with a characterization of exactness that looks very similar to Theorem~3.2 in~\cite{WJK} and to Lemma~3.3 in~\cite{JN}. Therefore, the following result cannot be called 'new', and we provide the proof just for the sake of completeness.

\begin{thr}\label{lemlem1}
Let $S$ be a semiring. The following are equivalent:

\noindent (E1) $S$ is left exact;

\noindent (E2) for any $\mathbf{A}\in S^{d\times n}$, $\mathbf{b}\in S^{d\times 1}$, the condition $\operatorname{left\,ker}\mathbf{A}\subseteq \operatorname{left\,ker}\mathbf{b}$ implies $\mathbf{b}\in\operatorname{right\,im}\mathbf{A}$.
\end{thr}

\begin{proof}
Assume (E1) is true, and let $\mathbf{A}\in S^{d\times n}$, $\mathbf{b}\in S^{d\times 1}$ be such that \begin{equation}\label{eqeq3}
\operatorname{left\,ker}\mathbf{A}\subseteq\operatorname{left\,ker}\mathbf{b}.
\end{equation}
We define the mapping $\varphi:\li\mathbf{A}\to S$ by $$\varphi\left(\sum\limits_{i=1}^d \lambda_i\mathbf{A}_i\right)=\sum\limits_{i=1}^d \lambda_i\mathbf{b}_i,$$
which is well defined because of~\eqref{eqeq3}. Since $\varphi$ is left $S$-linear, we can use the exactness of $S$ and obtain a left $S$-linear mapping $\psi:S^n\to S$ such that $\psi\left|\right._{\li\mathbf{A}}=\varphi$. Denoting $\alpha_i:=\psi(\mathbf{E}_i)$, we get
$$\mathbf{b}_i=\varphi(\mathbf{A}_i)=\psi\left(\sum\limits_{j=1}^n\mathbf{A}_i^j\mathbf{E}_j\right)=\sum\limits_{j=1}^n\mathbf{A}_i^j\alpha_j,$$ which implies $\mathbf{b}\in\operatorname{right\,im}\mathbf{A}$ and proves (E2).

Now we assume that (E2) is true, and we consider a finitely generated left semimodule $L\subseteq S^n$ and a left $S$-linear function $\varphi:L\to S$. We can write $L=\li\mathbf{A}$ for some matrix $\mathbf{A}$, and we define the vector $\mathbf{b}\in S^{d\times 1}$ by the formula $\mathbf{b}_i=\varphi(\mathbf{A}_i)$. (Here, the dimension $d$ is the number of rows of $\mathbf{A}$, or, equivalently, the number of generators of $L$.) The equation~\eqref{eqeq3} is true because $\varphi$ is well defined, so (E2) implies $\mathbf{b}\in\operatorname{right\,im}\mathbf{A}$. Therefore, we have $\mathbf{b}=\sum_j\mathbf{A}^j\alpha_j$ for some $\alpha_1,\ldots,\alpha_n\in S$, and then $\psi(x_1,\ldots,x_n)=x_1\alpha_1+\ldots+x_n\alpha_n$ is a mapping from $S^n$ to $S$ that coincides with $\varphi$ on $R$. We see that $S$ is left exact, so (E1) is true.
\end{proof}

Let us present an application of Theorem~\ref{lemlem1}. The following corollary presents a rather powerful condition that holds in all exact rings. This result seems to be new.

\begin{cor}\label{corcor1}
Any left exact semiring contains an element $e$ such that $1+1+e=1$.
\end{cor}

\begin{proof}
We consider the matrices
$$\mathbf{A}=
\begin{pmatrix}
0&1\\
1&1
\end{pmatrix},\,\,\,\,
\mathbf{b}=
\begin{pmatrix}
1+1\\
1
\end{pmatrix},$$
and we apply Theorem~\ref{lemlem1}.
The condition (E2) shows that either $\mathbf{b}\in\operatorname{right\,im}\mathbf{A}$ or  $\operatorname{left\,ker}\mathbf{A}\nsubseteq \operatorname{left\,ker}\mathbf{b}$. Let us treat these two cases separately.

\textit{Case 1.} If $\mathbf{b}\in\operatorname{right\,im}\mathbf{A}$, then there are $x_1,x_2\in S$ such that $x_2=1+1$ and $x_1+x_2=1$. We get $x_1+1+1=x_1+x_2=1$, which implies the desired conclusion.

\textit{Case 2.} If $\operatorname{left\,ker}\mathbf{A}\nsubseteq \operatorname{left\,ker}\mathbf{b}$, then there are vectors $\mathbf{u},\mathbf{v}\in S^{1\times 2}$ such that $\mathbf{u}\mathbf{A}=\mathbf{v}\mathbf{A}$ and
\begin{equation}
\label{eqeq1}\mathbf{u}\mathbf{b}\neq\mathbf{v}\mathbf{b}.
\end{equation}
The former condition shows that $\mathbf{u}^2=\mathbf{v}^2$, $\mathbf{u}^1+\mathbf{u}^2=\mathbf{v}^1+\mathbf{v}^2$, so we get
\begin{equation}\label{eqeq2}
\mathbf{u}^1+\mathbf{u}^1+\mathbf{u}^2=
\mathbf{u}^1+\mathbf{v}^1+\mathbf{v}^2=
\mathbf{u}^1+\mathbf{v}^1+\mathbf{u}^2=
\mathbf{v}^1+\mathbf{v}^1+\mathbf{v}^2,
\end{equation}
which is a contradiction. In fact, the left-hand side of~\eqref{eqeq1} coincides with the left-hand side of~\eqref{eqeq2}, and the right-hand side of~\eqref{eqeq1} coincides with the right-hand side of~\eqref{eqeq2}. Therefore, Case~2 is not an option, and the proof is complete.
\end{proof}

\begin{cor}\label{corcor2}
Any left exact semifield is either a ring or satisfies $1+1=1$.
\end{cor}

\begin{proof}
Let $e$ be the element as in Corollary~\ref{corcor1}. We get
$$(1+e)^2=1+e+e+e^2=1+e(1+1+e)=1+e,$$
so that $1+e=0$ or $1+e=1$. The former condition would imply that we have a ring, and the latter one shows that $1+1+e=1+1$ or $1+1=1$ again by Corollary~\ref{corcor1}.
\end{proof}

\section{A semifield version of Theorem~\ref{lemlem1}}

In this section we sharpen the condition (E2) in Theorem~\ref{lemlem1} under the additional assumption that $S$ is a semifield. Recall that a matrix $\mathbf{C}\in S^{n\times n}$ is \textit{invertible} if there exists a matrix $\mathbf{C}^{-1}$ such that $\mathbf{C}\mathbf{C}^{-1}=\mathbf{C}^{-1}\mathbf{C}=\mathbf{E}$.

\begin{obs}\label{obsobs1}
Let $\mathbf{C}\in S^{d\times d}$, $\mathbf{D}\in S^{n\times n}$ be invertible matrices, and let $\mathbf{A}\in S^{d\times n}$, $\mathbf{b}\in S^{d\times 1}$ be arbitrary. Then

\noindent (1) $\lk \mathbf{A}\subseteq\lk \mathbf{b} $ if and only if $\lk \mathbf{C}\mathbf{A}\mathbf{D}\subseteq\lk \mathbf{C}\mathbf{b}$,

\noindent (2) $\mathbf{b}\in \ri \mathbf{A}$ if and only if $ \mathbf{C}\mathbf{b}\in \ri \mathbf{C}\mathbf{A}\mathbf{D}$.
\end{obs}

\begin{proof}
Let us assume $(\mathbf{u},\mathbf{v})\in\lk \mathbf{A}\setminus \lk\mathbf{b}$, which means that $\mathbf{u}\mathbf{A}=\mathbf{v}\mathbf{A}$, $\mathbf{u}\mathbf{b}\neq\mathbf{v}\mathbf{b}$. We define
$\mathbf{u}'=\mathbf{u}\mathbf{C}^{-1}$, $\mathbf{v}'=\mathbf{v}\mathbf{C}^{-1}$, and we get
$$\mathbf{u}'\mathbf{C}\mathbf{A}\mathbf{D}=
\mathbf{u}\mathbf{C}^{-1}\mathbf{C}\mathbf{A}\mathbf{D}=
\mathbf{u}\mathbf{A}\mathbf{D}=\mathbf{v}\mathbf{A}\mathbf{D}=
\mathbf{v}\mathbf{C}^{-1}\mathbf{C}\mathbf{A}\mathbf{D}=
\mathbf{v}'\mathbf{C}\mathbf{A}\mathbf{D},$$
$$\mathbf{u}'\mathbf{C}\mathbf{b}=
\mathbf{u}\mathbf{C}^{-1}\mathbf{C}\mathbf{b}=
\mathbf{u}\mathbf{b}\neq \mathbf{v}\mathbf{b}=
\mathbf{v}\mathbf{C}^{-1}\mathbf{C}\mathbf{b}=
\mathbf{v}'\mathbf{C}\mathbf{b},$$
which means that $(\mathbf{u}',\mathbf{v}')\in\lk \mathbf{C}\mathbf{A}\mathbf{D}\setminus \lk\mathbf{C}\mathbf{b}$.
This proves the 'only if' direction of (1), and the 'if' direction follows as well by symmetry.

To prove (2), we note that a vector $\mathbf{w}$ satisfies $\mathbf{A}\mathbf{w}=\mathbf{b}$ if and only if the vector $\mathbf{w}'=\mathbf{D}^{-1}\mathbf{w}$ satisfies $\mathbf{C}\mathbf{A}\mathbf{D}\mathbf{w}'=\mathbf{C}\mathbf{b}$.
\end{proof}

Let us say that a matrix is \textit{column-stochastic} if the sum of elements in every column equals one. A semiring $S$ is called \textit{zero-sum free} if $a+b=0$ implies $a=b=0$ for all $a,b\in S$. We are ready to show that, in the case of zero-sum-free semifields, Theorem~\ref{lemlem1} remains true if we restrict the possible choices of $\mathbf{A}$ by column-stochastic matrices and the choices of $\mathbf{b}$ by vectors whose coordinates are zeros and ones.

\begin{cor}\label{obsobs2}
Let $S$ be a zero-sum free semifield. Then the condition (E2) in Theorem~\ref{lemlem1} is equivalent to the following:

\noindent (E2') for any column-stochastic matrix $\mathbf{A}\in S^{d\times n}$ and any vector $\mathbf{b}\in \{0,1\}^{d\times 1}$, the condition $\operatorname{left\,ker}\mathbf{A}\subseteq \operatorname{left\,ker}\mathbf{b}$ implies $\mathbf{b}\in\operatorname{right\,im}\mathbf{A}$.
\end{cor}

\begin{proof}
It is trivial that (E2) implies (E2'). To prove the opposite direction, assume that (E2) is not true. Then there are $\mathbf{A}\in S^{d\times n}$, $\mathbf{b}\in S^{d\times 1}$ such that $\operatorname{left\,ker}\mathbf{A}\subseteq \operatorname{left\,ker}\mathbf{b}$ and $\mathbf{b}\notin\operatorname{right\,im}\mathbf{A}$. The removal of zero columns of $A$ does not change these properties, so we can assume that every column of $A$ contains at least one non-zero entry. We define $\beta_i=\mathbf{b}_i$ if $\mathbf{b}_i\neq0$ and $\beta_i=1$ otherwise, and we set $\mathbf{C}$ to be the diagonal matrix with $\beta_1,\ldots,\beta_d$ on the diagonal. Further, we define $\alpha_j$ as the sum of the entries of the $j$th column of $\mathbf{C}^{-1}\mathbf{A}$. Since the semifield is zero-sum-free, the $\alpha_j$'s are non-zero, so we get an invertible matrix $\mathbf{D}$ if we put $\alpha_1,\ldots,\alpha_n$ on the diagonal and zeros everywhere else. Now we see that the matrix $\mathbf{C}^{-1}\mathbf{A}\mathbf{D}^{-1}$ is stochastic, the vector $\mathbf{C}^{-1}\mathbf{b}$ consists of zeros and ones, and the conditions $\operatorname{left\,ker}\mathbf{C}^{-1}\mathbf{A}\mathbf{D}^{-1}\subseteq \operatorname{left\,ker}\mathbf{C}^{-1}\textbf{b}$ and $\mathbf{C}^{-1}\mathbf{b}\notin \operatorname{right\,im}\mathbf{C}^{-1}\mathbf{A}\mathbf{D}^{-1}$ hold by Observation~\ref{obsobs1}. This shows that (E2') is not true.
\end{proof}

\section{Idempotent semifields are exact}

In this section we complete the proof of Theorem~\ref{thrthr}. Namely, we show that any semifield satisfying $1+1=1$ is necessarily left exact. In general, a semiring in which $1+1=1$ (or, equivalently, $x+x=x$ for all $x$) is called \textit{idempotent}. A natural (and very well known) ordering on an idempotent semiring is given as $x\geqslant y$ if and only if $x+y=x$. It is easy to see that the relation $\geqslant$ is a partial order compatible with the operations. In other words, the following result is true, see~\cite{Gol} for details.

\begin{obs}\label{obsord}
Let $S$ be an idempotent semiring and $p,q,r,s\in S$. Then

\noindent (1) $p\geqslant p$;

\noindent (2) If $p\geqslant q$, $q\geqslant r$, then $p \geqslant r$;

\noindent (3) If $p\geqslant q$, $q\geqslant p$, then $p=q$; 

\noindent (4) If $p\geqslant r$, $q\geqslant s$, then $p+q \geqslant r+s$ and $pq\geqslant rs$.
\end{obs}

We will write $p>q$ if $p\geqslant q$ and $p\neq q$. If $p\neq p+q$, then we write $q\nleqslant p$. The set of matrices or vectors over $S$ is still an idempotent semigroup with respect to addition, so these relations are applicable to matrices and vectors. Another obvious property of idempotent semirings is that they are zero-sum free.

\begin{obs}\label{obsord2}
Let $S$ be an idempotent semiring and $p,q\in S$. If $p+q=0$, then $p=q=0$.
\end{obs}

\begin{obs}\label{obsord3}
Let $S$ be an idempotent semifield containing at least three elements. Then there is an element $\lambda$ such that $\lambda\nleqslant 1$.
\end{obs}

\begin{proof}
Choose an arbitrary $a\notin\{0,1\}$. If $a\nleqslant 1$, then we are done, and otherwise we have $1+a=1$. This implies $a^{-1}+1=a^{-1}$, so we can take $\lambda=a^{-1}$.
\end{proof}

Before we proceed, we recall that a matrix $\mathbf{A}$ is called \textit{column-stochastic} if the sum of elements in every column of $\mathbf{A}$ equals one. A matrix is \textit{row-stochastic} if its transpose is column-stochastic.

\begin{lem}\label{lemidem1}
Let $S$ be an idempotent semifield containing at least three elements. Let $\mathbf{A}\in S^{d\times n}$ be a column-stochastic matrix that is not row-stochastic. Then there is a vector $\Lambda\nleqslant(1,\ldots,1)\in S^{1\times d}$ such that $\Lambda\mathbf{A}=(1,\ldots,1)$.
\end{lem}

\begin{proof}
We have $$\sum\limits_{i=1}^d\sum\limits_{j=1}^n \mathbf{A}_i^j=1+\ldots+1=1,$$
so that $\alpha_i:=\sum_{j=1}^n \mathbf{A}_i^j\leqslant 1$. If $\alpha_i=0$ for some $i$, then the $i$th row of $\mathbf{A}$ consists of zeros by Observation~\ref{obsord2}; in this case, we define $\Lambda$ as the vector whose coordinates are ones except the $i$th coordinate which is equal to the element $\lambda$ as in Observation~\ref{obsord3}. We have $\Lambda\nleqslant(1,\ldots,1)$ and $\Lambda \mathbf{A}=(1,\ldots,1)\mathbf{A}$; since $\mathbf{A}$ is column-stochastic, we get $(1,\ldots,1)\mathbf{A}=(1,\ldots,1)$ and complete the proof in our special case.

Now we assume that all of the $\alpha_i$'s are non-zero, and we define $\Lambda\in S^{1\times d}$ with $\Lambda^i$ being the inverse of $\alpha_i$. As said above, $\alpha_i\leqslant1$, so that $\Lambda^i\geqslant1$ by the item (4) of Observation~\ref{obsord}. Also, the assumption of the theorem states that $\mathbf{A}$ is not row-stochastic, which implies $\Lambda>(1,\ldots,1)$. Since $\mathbf{A}$ is column-stochastic, we have $(1,\ldots,1)\mathbf{A}=(1,\ldots,1)$, and using the item (4) of Observation~\ref{obsord}, we get
\begin{equation}\label{eqstoc1}\Lambda \mathbf{A}\geqslant (1,\ldots,1)\mathbf{A}=(1,\ldots,1).\end{equation}
Further, we get
$$\sum\limits_{j=1}^n\sum\limits_{i=1}^d \Lambda^i\mathbf{A}_i^j=\sum\limits_{i=1}^d \Lambda^i\alpha_i=1+\ldots+1=1,$$
which shows that
\begin{equation}\label{eqstoc2}\sum\limits_{i=1}^d \Lambda^i\mathbf{A}_i^j\leqslant 1\end{equation}
for all $j$. Putting the inequalities~\eqref{eqstoc1} and~\eqref{eqstoc2} together and using the item (3) of Observation~\ref{obsord}, we get $\Lambda \mathbf{A}=(1,\ldots,1)$. 
\end{proof}

\begin{lem}\label{lemidem2}
Let $S$ be an idempotent semifield containing at least three elements. Let $\mathbf{A}\in S^{d\times n}$ be a column-stochastic matrix and $\mathbf{b}\in\{0,1\}^{d\times 1}$ be a vector outside $\ri\mathbf{A}$. Then there are vectors $\mathbf{u},\mathbf{v}$ such that $\mathbf{u}\mathbf{A}=\mathbf{v}\mathbf{A}$ and $\mathbf{u}\mathbf{b}\neq \mathbf{v}\mathbf{b}$.
\end{lem}

\begin{proof}
Assume $\mathbf{b}$ contains $k$ ones. We can assume without loss of generality that the first $k$ coordinates of $\mathbf{b}$ are ones, and we write 
$$\mathbf{A}=
\left(\begin{array}{c|c}
P_{k\times (n-m)}&Q_{k\times m}\\\hline
R_{(d-k)\times (n-m)}&O_{(d-k)\times m}
\end{array}\right),\,\,\,\,
\mathbf{b}=
\begin{pmatrix}
J_{k\times 1}\\\hline
O_{(d-k)\times 1}
\end{pmatrix},$$
where $J$ is a $k\times 1$ vectors of ones, the $O$'s are zero matrices of relevant sizes, and $R$ has no zero column. 
Note that $k\neq0$ and $Q$ is not a row-stochastic matrix because otherwise we would have $\mathbf{b}\in\ri\mathbf{A}$. However, the equalities $k=d$, $m=0$, $m=n$ are possible, and they correspond to some blocks of the above matrices being empty.

By Lemma~\ref{lemidem1}, there is a vector $\Lambda\nleqslant(1,\ldots,1)\in S^{1\times k}$ such that $\Lambda Q=(1,\ldots,1)$. (If $k=d$ or $m=n$, this completes the proof immediately. If $m=0$, that is, if $Q$ is empty, then we choose a vector $\Lambda\nleqslant(1,\ldots,1)$ arbitrarily.)
Let $p\in S$ be the sum of all entries of the matrices $P$ and $\Lambda P$ plus one; let $r$ equal one plus the sum of the inverses of all non-zero entries of $R$. We set $M=(pr,\ldots,pr)\in S^{1\times(d-k)}$, and we get $MR\geqslant (p,\ldots,p)\in S^{1\times(m-n)}$. This means that $MR$ is greater than or equal to any row of $P$ and $\Lambda P$, so we get
$$(1,\ldots,1|M)\mathbf{A}=(MR|1,\ldots,1)=(\Lambda|M)\mathbf{A},$$
$$(1,\ldots,1|M)\mathbf{b}=1\nleqslant\Lambda^1+\ldots+\Lambda^k=(\Lambda|M)\mathbf{b},$$
which completes the proof.
\end{proof}

\begin{thr}\label{thrthr2}
Every idempotent semifield is left exact.
\end{thr}

\begin{proof}
Lemma~\ref{lemidem2} shows that an idempotent semifield $S$ possesses the property (E2') as in Corollary~\ref{obsobs2} whenever $S$ has at least three elements. Therefore, we can apply Theorem~\ref{lemlem1} to see that every such $S$ is left exact.

If $S$ contains $0$ and $1$ only, then the definitions allow us to identify $S$ uniquely as the binary Boolean semiring $\mathbb{B}$. The exactness of $\mathbb{B}$ can be proved by a routine application of Theorem~\ref{lemlem1}. Alternatively, one can get this result by applying a deeper one, Theorem~6.5 in~\cite{WJK}.
\end{proof}

The proof of Theorem~\ref{thrthr} is now complete. In fact, Corollary~\ref{corcor2} proves the 'only if' direction, and the 'if' direction follows from Theorem~\ref{thrthr2} and basic results of linear algebra which imply that any division ring is exact.

\section{A discussion and further work}\label{secfw}

We gave the complete characterization of semifields that are exact, but we do not know how to generalize this result to the case of arbitrary semirings. The arXiv version of this paper (\cite{myexac}) contained several questions regarding such characterizations, but later it turned out that several of these questions have already been answered. In particular, as Tran Giang Nam pointed out to the author, the solution to Problem~5.1 in~\cite{myexac} is negative. Namely, Example~4.16 and Theorem~4.18 in~\cite{AIK} give an example of an exact semiring that cannot be represented as a direct sum of a ring and an idempotent semiring. 
Nam also noticed that the semiring $B_3$ as in~Example 3.7 of~\cite{AIK} is selective and exact but is not a semifield, which solves Problem~5.2 in~\cite{myexac}.

Several interesting examples of idempotent semirings were examined in~\cite{Wil} by Wilding. Given a monoid $M$, he defines the semiring $\mathbb{B}M=(2^M, \cup, \cdot)$, where two subsets $A,B\subseteq M$ are multiplied as $AB=\{ab|a\in A, b\in B\}$. Wilding proves that $\mathbb{B}M$ is exact if $M$ is a group, and asks if the converse of this statement is true. Our approach is not sufficient to answer this question, and we believe that its resolution would lead to a significant step towards the classification of exactness in the idempotent case.

\section{Acknowledgements}

I would like to thank the anonymous reviewer for helpful comments and for pointing my attention to the papers~\cite{CSQ, LMS}. I am grateful to Grigory Garkusha for explaining to me why the exactness of a ring is equivalent to FP-injectivity and for pointing out that the result in~\cite{Shi} is essentially Theorem~3.2 in~\cite{Gar}. As pointed out in Section~\ref{secfw}, the first arXiv draft of this paper (see~\cite{myexac}) contained many questions that turned out to have been solved earlier by Tran Giang Nam and his colleagues in different papers on this topic. I would like to thank Nam for his comments on the corresponding results and for pointing my attention to relevant references.



\begin{thebibliography}{99}

\bibitem{AIK}
J. Y. Abuhlail, S. N. Il'in, Y. Katsov, T. G.  Nam, On V-semirings and semirings all of whose cyclic semimodules are injective, \textit{Communications in Algebra} 43.11 (2015) 4632--4654.

\bibitem{CSQ}
G. Cohen, S. Gaubert, J. P. Quadrat, Duality and separation theorems in idempotent semimodules, \textit{Linear Algebra Appl.} 379 (2004) 395--422.

\bibitem{Gar}
G. A. Garkusha, FP-injective and weakly quasi-Frobenius rings, \textit{Zap. S.-Peterburg Otd. Mat. Inst. Steklov (POMI)} 265 (1999) 110--129 (Russian). Engl. transl. in \textit{J. Math. Sciences} 112 (2002) 4303--4312.

\bibitem{Gol}
J. S. Golan. \textit{Semirings and their applications}, Springer, 2013.

\bibitem{JN}
M. Johnson, T. G. Nam, FP-injective semirings, semigroup rings and Leavitt path algebras, {Communications in Algebra} 45 (2017) 1893--1906.


\bibitem{LMS}
G. L. Litvinov, V. P. Maslov, G. B. Shpiz, \textit{Idempotent functional analysis: an algebraic approach}, \textit{Math. Notes} 69(5) (2001) 696--729.


\bibitem{NY}
W. K. Nicholson, M. F. Yousif, Quasi-Frobenius Rings, volume 158 of Cambridge Tracts in Mathematics, Cambridge University Press, Cambridge, 2003.

\bibitem{RS}
H. Rasiowa, R. Sikorski.  \textit{The mathematics of metamathematics}, Panstwowe Wydawnictwo Naukowe, Warsaw, 1963.

\bibitem{SpSt}
D. Speyer, B. Sturmfels, Tropical mathematics, \textit{Mathematics Magazine} 82 (2009) 163--173.

\bibitem{Shi}
Y. Shitov, Group rings that are exact, \textit{J. Algebra} 403 (2014) 179--184.

\bibitem{myexac}
Y. Shitov. 
Preprint (2016) arXiv:1609.09149v1.

\bibitem{Wil}
D. Wilding. \textit{Linear Algebra Over Semirings}, PhD dissertation, The University of Manchester, 2015.

\bibitem{WJK}
D. Wilding, M. Johnson, M. Kambites, Exact rings and semirings, \textit{J. Algebra} 388 (2013): 324--337.

\bibitem{Yan}
M. Yannakakis, Expressing combinatorial optimization problems by linear programs, \textit{Comput. System Sci.} 43 (1991) 441--466.


\end{thebibliography}
\end{document}